\documentclass[12pt]{amsart}

\usepackage{latexsym}
\usepackage{amsmath}
\usepackage{amsfonts}
\usepackage{amssymb}
\usepackage{stackengine}

\usepackage{mathtools}
\usepackage{verbatim}
\usepackage{xcolor}

\usepackage{array}
\newcolumntype{C}[1]{>{\centerring\arraybackslash}p{#1}}
\usepackage{booktabs,tabularx}
\usepackage{enumitem}
\usepackage{microtype}

\DeclarePairedDelimiter\ceil{\lceil}{\rceil}
\DeclarePairedDelimiter\floor{\lfloor}{\rfloor}

\newtheorem{theorem}{Theorem}[section]
\newtheorem*{theorem*}{Theorem}
\newtheorem{lemma}[theorem]{Lemma}
\newtheorem{corollary}[theorem]{Corollary}
\newtheorem{proposition}[theorem]{Proposition}

\theoremstyle{definition}
\newtheorem{definition}[theorem]{Definition}

\theoremstyle{remark}
\newtheorem{remark}[theorem]{Remark}

\numberwithin{equation}{section}

\begin{document}

\title{Congruences for a class of eta-quotients and their applications}

\author{Shashika Petta Mestrige}
\address{Mathematics Department\\
Louisiana State University\\
Baton Rouge, Louisiana}
\email{pchama1@lsu.edu}

\date{Sep 11, 2020}

\begin{abstract}
The partition function $ p_{[1^c\ell^d]}(n)$ can be defined using the generating function,
\[\sum_{n=0}^{\infty}p_{[1^c{\ell}^d]}(n)q^n=\prod_{n=1}^{\infty}\dfrac{1}{(1-q^n)^c(1-q^{\ell n})^d}.\]
In \cite{P}, we proved infinite family of congruences for this partition function for $\ell=11$. In this paper, we extend the ideas that we have used in  \cite{P} to prove infinite families of congruences for the partition function $p_{[1^c\ell^d]}(n)$ modulo powers of $\ell$ for any integers $c$ and $d$, for primes $5\leq \ell\leq 17$. This generalizes Atkin, Gordon and Hughes' congruences for powers of the partition function. The proofs use an explicit basis for the vector space of modular functions of the congruence subgroup  $\Gamma_0(\ell)$. Finally we use these congruences to prove  congruences and incongruences of the $\ell$-colored generalized Frobenius partitions, $\ell-$regular partitions, and $\ell-$core partitions for $\ell=5,7,11,13$ and $17$.

\end{abstract}

\maketitle
\section{Introduction}
\label{intro}

An (integer) partition of $n$ is a non-increasing sequence of positive integers $\lambda_1 \geq \lambda_2 \cdots \geq \lambda_r \geq 1$ that sum to $n$. Let $p(n)$ be the number of partitions of $n$. By convention, we take  $p(0)=1$ and $p(n)=0$ for negative $n$.

This function has been extensively studied since the last century. In the 1920's
Ramanujan discovered amazing congruence properties for $p(n)$.

 \begin{theorem*}[Ramanujan, Watson, Atkin]\label{T:1.0}
    For all positive integers $j$, we have,
    \begin{align*}
p(5^jn+\delta_{5,j}) & \equiv0\pmod{5^j}, \\
p(7^jn+\delta_{7,j}) & \equiv0\pmod{7^{[\frac{j+2}{2}]}}, \\
p(11^jn+\delta_{11,j}) & \equiv0\pmod{11^j},
\end{align*}
    where  $24\delta_{\ell,j}\equiv1\pmod{\ell^j}$ for $\ell \in \{5,7,11\}$.
\end{theorem*}

Ramanujan in \cite{RS1}  proved the first two congruences for the case of $j=1$  by using the Jacobi triple product and later in \cite{RS2} by using the theory of modular forms on $SL_2(\mathbb{Z})$.  He conjectured that $p(7^jn+\delta_{7,j})$ is also divisible by $7^j$. In $1935$, Gupta showed that it is incorrect for $j=3$. However,  proofs for arbitrary $j$ for the first two congruences  were given in $1938$ by Watson in \cite{W} using modular equations for prime $5$ and $7$. Ramanujan in \cite{RS3} stated that he found a proof for the third congruence for $j=1,2$, but did not include the proof. In $1967$, Atkin proved the third congruence in \cite{A2}.\\

These fascinating congruence properties not only hold for the partition function itself, but also expected to hold for generalized partitions. We study a two-parameter family of partition generating functions $p_{[1^c\ell^d]}(n)$. This partition function is also well studied in recent years, for example see Chan and Toh \cite{CHT}, and Wang \cite{WL}. The partition function $p_{[1^c\ell^d]}(n)$ is defined using the generating function in the following way.
\begin{equation*}
\prod_{n=1}^{\infty}\dfrac{1}{(1-q^n)^c(1-q^{\ell n})^d}=\sum_{n=0}^{\infty}p_{[1^c\ell^d]}(n)q^n.
\end{equation*}
Generalizing partition function results to other functions had been a main research topic since 1960's. The first attempt was to study the multi-partitions or $k-$color partitions which are the coefficients of the $k^{th}$ power of the generating function of the partition function. In $1960$'s Atkin in \cite{A1}, proved congruences for $k-$ color partitions for small primes using modular equations. Notice that we can obtain the generating function for $k-$color partitions by setting $k=c,d=0$. See \cite{MI} for a recent study of the congruences of this partition function. \\

In \cite{P}, we proved a unified way to prove congruences for this partition function modulo powers of $11$, for all $c,d$ using Gordon's approach to prove congruences for $k-$color partitions. In this paper we revisit Atkin's proof and extend his result for $p_{[1^c\ell^d]}(n)$ for primes $5,7,$ and $13$.  We also proved congruences for $\ell=17$ using the work of Kim Hughes on $k-$color partitions modulo powers of $17$ in \cite{H}. 
 
\begin{theorem}\label{T:1.1}
For $\ell=5,7,13,17$, for any integers $c$, $d$ and for any positive integer $r$,
\begin{equation}\label{T:1}
    p_{[1^c{\ell}^d]}({\ell}^rm+n_r)\equiv 0\pmod{{\ell}^{A_r}},
\end{equation}
where $24n_r\equiv(c+{\ell}d)\pmod{{\ell}^r}$.
Here $A_r$  only depends on the integers $c,d$ and it can be calculated explicitly.
\end{theorem}
Moreover, we obtain the following corollary, this is  similar to Gordon's Theorem 1.1 in \cite{GB}.
\begin{corollary}\label{C:1.3}
For $\ell=5,7,$ and $13$ and for any positive integer r,
\[
 p_{[1^c\ell^d]}(n)\equiv 0\pmod{\ell^{\frac{1}{2}\alpha_{\ell} r +\epsilon}},
 \]
 where $24n\equiv(c+\ell d)\pmod{\ell^r}$, $\epsilon=\epsilon(c,d)=O(log|c+\ell d|)$ and when $c+\ell d> 0$ , $\alpha_{\ell}$ 
 depends of the residue of $c+\ell d\pmod{24}$ which is shown in the following table,
 \begin{table}[htp]
\centering
\footnotesize\setlength{\tabcolsep}{2.3pt}
\begin{tabular}{l@{\hspace{6pt}} *{25}{c}}
\cmidrule(l){2-25}
& 1 & 2 & 3 & 4 & 5 & 6 & 7 & 8 & 9 & 10 & 11 & 12 & 13 & 14 & 15 & 16 & 17 & 18 & 19 & 20 & 21 & 22 & 23 & 24 \\
\midrule
\bfseries $\ell=5$
 & 2 & 1 & 1 & 1 & 2 & 2 & 1 & 1 & 1 & 1 & 1 & 0 & 0 & 0 & 1 & 1 & 0 & 0 & 0 & 1 & 1 & 0 & 0 & 0\\
\bfseries $\ell=7$
& 1 & 1 & 1 & 2 & 1 & 1 & 1 & 0 & 0 & 0 & 1 & 0 & 0 & 1 & 0 & 0 & 0 & 1 & 0 & 0 & 1 & 0 & 0 & 0\\
\bfseries $\ell=13$
& 0 & 0 & 0 & 0 & 0 & 0 & 0 & 1 & 0 & 0 & 0 & 0 & 0 & 0 & 0 & 0 & 0 & 0 & 0 & 0 & 0 & 0 & 0 & 0\\
\bottomrule
\addlinespace
\end{tabular}
\caption{Values of $\alpha_{\ell}$}\label{t:1.3}
\end{table}
\newline for ${c+\ell d<0}$, the entries in the last column need to changed to $1$ for $\ell=5,7$.
\end{corollary}
\begin{remark}
This is the same shape as Gordon's result for $k-$colored partitions, with k replaced by $c+\ell d$. See \cite{P} for a similar result for $\ell=11.$
\end{remark}
Similar to the prime 11 case, we can obtain the following result for $p_{[1^c17^d]}(n)$ using \cite{H}.
\begin{corollary}\label{C:1.5}
For any positive integer $r$, let $c,d$ be integers such that, $24n_r\equiv(c+17 d)\pmod{17^r}$. Then we have $p_{[1^c17^d]}(17^rm+n_r)\equiv 0\pmod{17^{\frac{1}{2}\alpha_{17} r +\epsilon}}$, where $\epsilon=\epsilon(c,d)=O(\log|c+17d|)$  and when $c+17d\geq 0$ , $\alpha_{17}$ depends on the residue of $c+17d\pmod{96}$ which is shown in table \ref{t:1.5}.
\begin{table}[htp]
\centering
\footnotesize\setlength{\tabcolsep}{2.3pt}
\begin{tabular}{l@{\hspace{6pt}} *{25}{c}}
\cmidrule(l){2-25}
& 1 & 2 & 3 & 4 & 5 & 6 & 7 & 8 & 9 & 10 & 11 & 12 & 13 & 14 & 15 & 16 & 17 & 18 & 19 & 20 & 21 & 22 & 23 & 24 \\
\midrule
\bfseries 0
 & 0 & 0 & 1 & 0 & 0 & 0 & 0 & 0 & 0 & 0 & 0 & 0 & 0 & 0 & 0 & 0 & 0 & 0 & 0 & 1 & 0 & 0 & 0 & 0\\
\bfseries 24
& 0 & 0 & 0 & 0 & 0 & 0 & 0 & 0 & 0 & 0 & 0 & 1 & 0 & 0 & 0 & 0 & 0 & 0 & 0 & 0 & 0 & 0 & 0 & 0\\
\bfseries 48
& 1 & 1 & 1 & 1 & 1 & 2 & 1 & 0 & 0 & 0 & 0 & 0 & 0 & 0 & 0 & 0 & 1 & 0 & 0 & 0 & 0 & 0 & 0 & 0\\
\bfseries 72
& 0 & 0 & 0 & 0 & 0 & 0 & 0 & 0 & 0 & 1 & 0 & 0 & 0 & 0 & 0 & 0 & 0 & 0 & 0 & 0 & 0 & 0 & 0 & 0\\

\bottomrule
\addlinespace
\end{tabular}
\caption{Values of $\alpha_{17}$}\label{t:1.5}
\end{table}

Here the entry is $\alpha(24i+j)$ where row labelled $24i$ and column labeled $j$. When $c+17d<0$, the last column must be changed to $0, 2, 0, 0$.
\end{corollary}

Even though Ramanujan's congruences and in general congruences between coefficients of modular forms were studied for so many years, a little attention has given to incongruences between them. Proving incongruences between modular forms has been an interesting problem in recent years. Garthwaite and Jameson recently proved incongruences between a large class of modular forms in \cite{GJ} using the $q-$ expansion principle developed by P. Deligne and M. Rapport. Here we use the expansions of generating modular functions for $p_{[1^c\ell^d]}(n)$ in terms of the explicit basis elements to determine incongruences.
\begin{corollary}
\label{C1.5} Let $A_r$ and $n_r$ be the integers mentioned in theorem \ref{T:1}.  If $A_r=0$ then for $\ell=5,7,$ and $13$, there is some integer $m$ such that,
\begin{equation}
    p_{[1^c\ell^d]}(\ell^r m+n_r)\not\equiv0\pmod{\ell}.
\end{equation}
 This result holds for $r=1$ when $\ell=11$ and $17$.
\end{corollary}

\begin{remark}
For $\ell=11$, we expect the incongruences should be hold for all positive integer $r$ since $A_r$ is the best possible bound as shown in \cite{GB}, but we do not prove it in this paper.
\end{remark}

In $1984$s George Andrews in \cite{AG} introduced the $k$-colored generalized frobenius partitions $C\Phi_k(q)$. 
 He showed that these partitions can be defined using the following generating function.
 \begin{equation*}
   \begin{aligned}
C\Phi_k(q):=&\prod_{n=1}^{\infty}\frac{1}{(1-q^n)^k}\sum_{m_1,m_2,\cdots ,m_k\in \mathbb{Z}}q^{Q(m_1,m_2,\cdots,m_k)}=\sum_{n=0}^{\infty}c\phi_k(n)q^n,\\
&\text{where}\;\;Q(m_1,m_2,\cdots,m_k)=\sum_{i=1}^{k-1}m_i^2+\sum_{1\leq i \leq j\leq k-1}m_im_j.
   \end{aligned}  
 \end{equation*}
 See \cite{GS} and \cite{KL} for interesting studies about this partition function. In $2018$ Chan, Wang, and Yang in \cite{CWY} studied this partitions using the theory of modular forms and derived new representations. Using their work and theorem \ref{T:1.1} we proved the congruences for $c\phi_k(n)$ for $k=5,7,$ and $11$.

    \begin{corollary}\label{C:1.7}
    For each non negative integer $r$ and for all $m$ we have,
   \begin{equation}\label{E:1.6} c\phi_5\left(5^{2r}m+\frac{19\cdot5^{2r}+5}{24}\right)\equiv 0 \pmod{5^{2r-1}}.
   \end{equation}

   \begin{equation}\label{E:1.7}
   c\phi_{7}\left(7^{2r}m+\frac{17\cdot 7^{2r}+7}{24}\right)\equiv 0\pmod{7^{r}}.
   \end{equation}

   \begin{equation}\label{E:1.8}
   c\phi_{11}\left(11^{2r}m-\frac{ 11^{2r+1}-11}{24}\right)\equiv 0\pmod{11^{2r-1}}.
   \end{equation}
    \end{corollary}

   We also use corollary \ref{C:1.5} to prove an incongruence for $k=13$.
   \begin{corollary} For each each positive integer $r$ and for some integer $m$, we have,
    \begin{equation}\label{E:1.9}
   c\phi_{13}\bigg(13^{2r}m-\frac{13-13^{2r+1}}{24}\bigg)\not\equiv 0\pmod{13}.
   \end{equation}
   \end{corollary}

To prove these results, we use modular equations for small primes. In section \ref{ME}, we explain  how these equations are proved.\\

In section \ref{AM}, we use modular equations to derive formulas for important modular functions that we have defined in \eqref{E.Sr}. Here we defined these functions more general way compared to \cite{K}. These equations were obtained by using the modified versions of modular equations.\\

In section \ref{CMF}, we construct modular functions $\Gamma_0(\ell)$ such that the part of these functions are the generating functions for the partition function $p_{[1^c\ell^d]}(n)$. \\

In section \ref{P}, the proofs of the theorem \ref{T:1.1} and corollaries  \ref{C:1.3}, \ref{C:1.5} are given. These proofs are similar to the proofs given in \cite{P}, however here, we give a new proof for incongruence \eqref{E:1.9}.\\

In the last section, we use theorem \ref{T:1.1} to prove congruences and incongruences for $\ell-$regular partitions and $\ell-$core partitions for primes $5,7, 11, 13$ and $17$.

\section{Modular forms and modular equations}\label{ME}

Modular forms have played a major role in modern number theory, including Wiles' celebrated proof of Fermat's Last Theorem.  Modular forms are also appear in other areas such as algebraic topology, sphere packing and string theory. See \cite{DS} and \cite{KN} for basic definitions, theories and applications.\\

Modular forms have also been used extensively to prove partition congruences. This is due to the generating functions of  partitions in some cases are modular forms or can be related to modular forms. See \cite{GF} for Garvan's work on  proving congruences using modular forms for a special type of partitions modulo $5,7,$ and $13$. \\

There are some available tools to construct modular forms or functions explicitly. One of them is the Dedekind eta function, $\eta(z)$. This is a weight $1/2$ modular form with a certain multiplier system on $SL_2(\mathbb{Z})$ and it is closely related to the generating function of the partition function.
See chapter $4$, \cite{K} for more details.
\[\eta(z):=q^{\frac{1}{24}}\prod_{n=1}^{\infty}(1-q^n).\]
  Now we use the Dedekind eta function to define two modular functions which play a significant role in our work.   Let $g_{\ell,r}(z)$ be defined in $\mathbb{H}$ by
\[g_{\ell,r}(z) :=\left\{\frac{\eta(\ell z)}{\eta (z)}\right\}^r.\]
\begin{theorem}[Theorem $1$, chapter $7$, \cite{K}]

If $\ell$ is a prime greater than 3 and $r$ is an integer such that  $r(\ell -1)\equiv 0\pmod{24}$, then $g_{\ell,r}(z)$ is a modular function on $\Gamma_0(\ell).$
\end{theorem}
Let $\ell$ be a prime greater than $3$, let $\phi_{\ell}(z)$ be defined in $\mathbb{H}$ by
\[\phi_{\ell}(z) :=\frac{\eta(\ell^2 z)}{\eta (z)}.\]
\begin{theorem}[Theorem $2$, Chapter $7$, \cite{K}]

 If $\ell$ is a prime greater than $3$ then $\phi_{\ell}(z)$ is a modular function on $\Gamma_0(\ell^2).$
\end{theorem}
Now let $g_{\ell}(z)$ be the modular form $g_{\ell,r}(z)$ where $r$ is the least positive integer satisfying the condition $r(\ell -1)\equiv 0\pmod{24}$.

\begin{definition}
 By the {\it modular equation} for prime $\ell$, we refer to an algebraic equation connecting two modular forms $g_{\ell}(z)$ and $\phi_{\ell}(z)$.
\end{definition}

\begin{theorem}
For $\ell=5,7, 13$ these equations are ,\\

For $\ell=5$,
    \begin{equation}\label{E:5}
    \phi_5^5(z)=g_5(5z)\left(5^2\phi_5^4(z)+5^2\phi_5^3(z)+5\cdot 3\phi_5^2(z)+5\phi_5(z)+1\right).
    \end{equation}
\vspace{2mm}

For $\ell=7$,
    \begin{equation}\label{E:7}
  \begin{aligned}
    \phi_7^7(z)
    &= g_7^2(7z)\{343\phi_7^6(z)+343\phi_7^5(z) + 147\phi_7^4(z)+ 49\phi_7^3(z)+21\phi_7^2(z)\\
    &+7\phi_7(z)+1\}
    +g_7(7z)\{7 \phi_7^4(z)+35\phi_7^5(z)+49\phi_7^6(z)\}.
  \end{aligned}
\end{equation}
\vspace{2mm}

 For $\ell=13$,
    \begin{equation}\label{E:13}
  \phi_{13}^{13}(z)+\sum_{r=1}^{13}\sum_{p=\floor{(r+2)/2}}^7m_{r,p}g_{13}^{p}(13z)\phi_{13}^{13-r}(z)=0.
\end{equation}

\begin{equation*}
\left(m_{r,p}\right)=
\left[\begin{smallmatrix}
-11\cdot 13 & -36\cdot13^2 & -38\cdot13^3 & -20\cdot 13^4 & -6\cdot 13^5 & -13^6 &-13^6 \\
 & 204\cdot 13 & 346\cdot 13^2 & 222\cdot 13^3 & 74\cdot 13^4 & 13^6 &13^6  \\
 & -36\cdot 13 & -126\cdot 13^2  & -102\cdot 13^3 & -38\cdot 13^4 & -7\cdot 13^5 & -7\cdot 13^5\\
& & 346\cdot 13 & 422\cdot 13^2 & 184\cdot 13^3 & 37\cdot 13^4 & 3\cdot 13^5\\
& & -38\cdot 13 & -102\cdot 13^2 & -56\cdot 13^3 & -13^5 & -15\cdot 13^4\\
& & & 222\cdot 13 & 184\cdot 13^2 & 51\cdot 13^3 & 5\cdot 13^4\\
& & & -20\cdot 13 & -38\cdot 13^2 & -13^4 & -19\cdot 13^3\\
& & & & 74\cdot 13 & 37\cdot 13^2 & 5\cdot 13^3\\
& & & & -6\cdot 13 & -7\cdot 13^2 & -15\cdot 13^2\\
& & & & & 13^2 & 3\cdot 13^2\\
& & & & & -13 & -7\cdot 13\\
& & & & & & 13\\
& & & & & & -1
\end{smallmatrix}\right]
\end{equation*}

\end{theorem}

   For primes $2,3,5,7$, the modular equation can be found in Atkin's paper \cite{A1}. The modular equation for prime $13$ can be found Atkin and O'Brien paper \cite{AO}.

   In \cite{GB}, Gordon, and in \cite{H}, Hughes found similar equations for $\ell=11$ and $\ell=17$ respectively. 
   However, these equations are a little more complicated than the above equations since the genus being one means that there does
not exist one single generator.
   
   \begin{theorem}[Riemann-Hurwitz formula, \cite{HR}]
   Let $X, Y$ be two compact Riemann surfaces, let $\pi:Y\rightarrow X$ be a covering map and let $N$ be the index of $Y$ in $X$. Then,
   \[2g(Y)-2=N\cdot(2g(X)-2)+\sum_{p\in Y}(e_p-1).\]
   where $g(Y),g(X)$ are the genera of $Y,X$ respectively and $e_p$ is the ramification degree at $p$.
   \end{theorem}
   
   This is an extremely useful theorem to calculate the genus of modular curves. For example, the genus of $X_0(5)$ is $0$. This follows from the fact that $\Gamma_0(5)$ has $2$ cusps with orders $1,5$, two elliptic points of order $2$ ,the genus of $X_0(1)$ is $0$ and the index of $X_0(5)$ in $X_0(1)$ is $6$. It can be proved similarly that the genus of $X_0(7),X_0(13)$ are $0$ and the genus of $X_0(11), X_0(17)$ are $1$. \\
   
\subsection{Proving modular equations for $5,7$ and $13$.}

The modular equations stated here can be proved using the ramification data of cusps of the modular curves $X_0(\ell^2)$ for $\ell=5,7,13$. Here we only explain how to obtain the equation for prime $5$, since the method for proving the other modular equations is similar.\\

The modular curve $X_0(25)$ has $6$ cusps $0,\infty,\frac{1}{5},\frac{2}{5},\frac{3}{5},\frac{4}{5}.$ with cusp widths $25,1,1,1,1,1$ respectively. Here $\phi_5(z)=\frac{\eta(25z)}{\eta(z)}$ and $g_5(5z)=\left(\frac{\eta(25z)}{\eta(5z)}\right)^6$ are modular functions on $\Gamma_0(25)$. Using theorem 1.65 in \cite{O}, we calculate the order of vanishing of these modular functions at cusps and they are stated in table \ref{t:2}. \\

\begin{table}[htp]
    \begin{tabular}{|l@{\hspace{20pt}}|*{7}{c|}}
\hline
 & $0$ & $\infty$  & $\frac{1}{5}$ & $\frac{2}{5}$ & $\frac{3}{5}$ &  $\frac{4}{5}$ \\
\cline{1-7}
\hline
\hline
$\phi_5(z)$  & $-1$ & $1$ & $0$ & $0$ & $0$ & $0$\\
\hline
\cline{1-5}
$g_5(5z)$ & $-1$ & $5$   & $-1$ & $-1$ & $-1$ & $-1$\\
\hline
\end{tabular}\\
\vspace{2mm}
    \caption{}\label{t:2}
    \end{table}

We first consider $\frac{1}{g_5(z)}$. This function has only a pole of order $5$ at $\infty$ and holomorphic at every point on the modular curve $X_0(25)$. Now by subtracting relevant negative powers of $\phi_5(\tau)$, we remove the pole at $\infty$. The resulting function,

\[\frac{1}{g_5(5z)}-\frac{1}{\phi_5^5(z)}-\frac{5}{\phi_5^4(z)}-\frac{15}{\phi_5^3(z)}-\frac{5^2}{\phi_5^2(z)}-\frac{5^2}{\phi_5(z)}\]

does not have any poles on the modular curve $X_0(25)$. Hence by Liouville's theorem, it is a constant.

\section{Applications of modular equations}\label{AM}
\begin{theorem}[Newton's Formula, theorem $9$, chapter 7, \cite{K}]\label{T:3.1}

Let \[f(x)=x^q-p_1x^{q-1}+p_2x^{q-2}-\cdots+(-1)^qp_q,\]

with roots $\phi_1,\cdots ,\phi_q$. For h a positive integer put, $S_h=\sum_{i=1}^q\phi_i^h$. Then if $1 \leq h \leq q$, 
\[ S_h-p_1S_{h-1}+p_2S_{h-2}-\cdots+(-1)^{h-1}p_{h-1}S_1+(-1)^hp_hh=0.\]
if $h>q$,
\[S_h-p_1S_{h-1}+p_2S_{h-2}-\cdots +(-1)^qp_qS_{h-q}=0.\]
\end{theorem}

As described in \cite{K}, using the notation of Newton's formula, we let $S_{r,\ell}$ be the sum of the $r^{th}$ power of the roots of modular equation for prime $\ell$,

\begin{equation}\label{E.Sr}
S_{r,\ell}=\sum_{k=0}^{\ell-1}\phi_{\ell}^r(\zeta_{\ell}^kq^{\frac{1}{\ell}}).
\end{equation}

For a Laurent series $f(z)=\sum_{n\geq N} a(n)q^n$, we define the $U_p$ operator by,
\begin{equation}\label{E:2.1}
U_p\left(f(z)\right)=\sum_{pn\geq N} a(pn)q^n.
\end{equation}
\vspace{3mm}

Let $g(z)=\sum_{n\geq N} b(n)q^{n}$ be another Laurent series. The following simple property will play a key role in our proof.
\begin{equation}\label{E:2.2}
U_p\left(f(z)g(pz)\right)=g(z)U_p\left(f(z)\right).
\end{equation}

\begin{proposition}[\cite{AL}, lemma 7]\label{P:Gamma0N}
If $f(z)$ is a modular function for $\Gamma_0(N)$, if $p^2|N $, then $U_p\Big{(}f(z)\Big{)}$ is a modular function for $\Gamma_0(N/p)$.
\end{proposition}

Let $h(z)$ be a modular function on $\Gamma_0(\ell)$. As shown in \cite{K} lemma $2$, chapter $8$, we have

\begin{equation}\label{E:2.1.2}
    U_{\ell}(h(z))=\frac{1}{\ell}\sum_{k=0}^{\ell-1}h\bigg(\zeta_{\ell}^kq^{\frac{1}{\ell}}\bigg).
\end{equation}

Now combining equations \eqref{E.Sr} and \eqref{E:2.1.2},

\begin{equation}\label{E:Sr}
    S_{r,\ell}(z)=\ell\cdot U_{\ell}\bigg(\phi_{\ell}^r(z)\bigg).
\end{equation}

We use the modular equations to find Laurent series expansions for $S_{r,\ell}(z)$ in $g_{r,\ell}(z)$. Lemma \ref{L:3.2}-\ref{L.3.5} describe the $\ell$-adic orders of the coefficients in those expansions. For a rational integer a, let $\pi_{\ell}(a)$ to be the  $\ell$-adic order of $a$.

\begin{lemma}\label{L:3.2}
Let $r$ be a non zero integer. Then $S_{r,5}(z)$ is a Laurent polynomial in $g_5(z)$ of the form $S_{r,5}(z)=\sum_{p=-\infty}^{\infty}a_{r,p}^5g_5(z)^p,$
where $a_{r,p}^5$ is an integer divisible by 5.
\begin{equation}\label{E:a5}
\pi_5(a_{r,p}^5)\geq \left\lfloor\frac{5p-r+1}{2}\right\rfloor,
\end{equation}

 for $r>0$, $a_{r,p}^5\not=0$ for $\floor{\frac{r+4}{5}}\leq p \leq r$ and for $r<0$, $a_{r,p}^5=0$ for $\floor{\frac{r+4}{5}}>p$.
 
\end{lemma}
\begin{proof}
For $r>0$ this inequality can be found in \cite{K}.
We prove this inequality for $r<0$ here. To calculate the $S_{r,5}$ when $r<0$ we modify the modular equation by dividing it by respective power of $\phi_5(\tau)$ and $g_5(z)$.
\begin{equation*}
\begin{aligned}
\bigg(\phi_5^{-1}(z)\bigg)^5+5\bigg(\phi_5^{-1}(z)\bigg)^4
&+5\cdot3\bigg(\phi_5^{-1}(z)\bigg)^3+5^2\bigg(\phi_5^{-1}(z)\bigg)^2\\
&+5^2\bigg(\phi_5^{-1}(z)\bigg)-\bigg(g_5^{-1}(5z)\bigg)=0.
\end{aligned}
\end{equation*}
Then using Newton formula, we have

\begin{enumerate}
    \item $S_{-1,5}(z)=-5,$
    \item $S_{-2,5}(z)=-5,$
    \item $S_{-3,5}(z)=25,$
    \item $S_{-4,5}(z)=-25,$
    \item $S_{-5,5}(z)=5g_5^{-1}(z),$\item $S_{r,5}(z)$=$-5S_{r+1,5}(z)-15S_{r+2,5}(z)-25S_{r+3,5}(z)-25S_{r+4,5}(z)+g^{-1}(z)S_{r+5,5}(z).$
\end{enumerate}
\vspace{3mm}
First of all using induction with the help of the recurrence relation 
satisfied by the $S_{r,5}$ we see that for $r<0$, $a_{r,p}^5=0$ for  $\floor{\frac{r+4}{5}}> p$.\\ 

Now notice that above claim is true for $r=-1\cdots-5$. Now we assume the claim hold for any $r<-5$. Again from the recurrence relation we have,
\[a_{r,p}=-5 a_{r+1,p}-15a_{r+2,p}-25a_{r+3,p}-25a_{r+4,p}+a_{r+5,p+1}\]

Now by assumption we have that the 5-adic valuation of the left hand side satisfies required inequality. 

\begin{equation*}
    \begin{aligned}
    \pi_5\left(a_{r,p}\right)\geq&{\rm min}\{\pi_5(a_{r+1,p})+1, \pi_5(a_{r+2,p})+1, \pi_5(a_{r+3,p})+2\\
    ,&\pi_5(a_{r+4,p})+2, \pi_5(a_{r+5,p+1})\}\\
    =&\left\lfloor\frac{5p-r+1}{2}\right\rfloor.
    \end{aligned}
\end{equation*}
\end{proof}
Notice that we have used $a_{r,p}$ to denote $a_{r,p}^5$.

\begin{lemma}\label{L:3.4}
Let $r$ be a non zero integer. Then $S_{r,7}(z)$ is a Laurent polynomial in $g_7(z)$ of the form $S_{r,7}(z)=\sum_{p=-\infty}^{\infty}a_{r,p}^7g_7^p(z)$,
where $a_{r,p}^7$ is an integer divisible by 7,
\begin{equation}\label{E:a7}
\pi_7(a_{r,p}^7)\geq \left\lfloor\frac{7p-2r+3}{4}\right\rfloor,
\end{equation}
for $r>0$, $a_{r,p}^7\not=0$ for $\floor{\frac{2r+6}{7}}\leq p \leq 2r$. For $r<0$, $a_{r,p}^7=0$ for $\floor{\frac{2r+6}{7}}>p .$ 
\end{lemma}

\begin{proof}
For $r>0$ this inequality can be found in \cite{K}. We here prove this inequality when $r<0$.  
To calculate the $S_{r,7}(z)$ when $r<0$ we modify the modular equation by dividing it by respective power of $\phi_7(z)$ and $g_7(z)$ 
 
Then using Newton's formula, we have 
\begin{enumerate}[]
    \item $S_{-1,7}(z)=-7,$
    \item $S_{-2,7}(z)=7,$
    \item $S_{-3,7}(z)=-7^2,$
    \item $S_{-4,7}(z)=-7^2-2^2\cdot 7 g^{-1}_7(z),$
    \item $S_{-5,7}(z)= 7^3+10\cdot 7 g^{-1}_7(z)$,
    \item $S_{-6,7}(z)=7^3,$
    \item $S_{-7,7}(z)= 7^3g^{-1}_7(z)+7g^{-2}_7(z).$
    
\vspace{2mm}
For $r\leq -8$,
    \begin{equation*}
        \begin{aligned}
        S_{r,7}(z)=&-7S_{r+1,7}(z)-3\cdot 7S_{r+2,7}(z)-7^2S_{r+3}(z)\\
        &-(3\cdot7^2+7g^{-1}_7(z))S_{r+4,7}(z) -(7^3+5\cdot 7)S_{r+5}(z)\\
        &-(7^3+7^2g^{-1}_7(z))S_{r+6,7}(z)+g^{-2}_7(z)S_{r+7,7}(z).
        \end{aligned}
    \end{equation*}
\end{enumerate}

First of all we can use the recurrence relation satisfied by the $S_{r,7}(z)$ to see that for $r<0$, $a_{r,p}^7=0$ for $\floor{\frac{2r+6}{7}}>p$.
Now notice that above claim is true for $r=-1\cdots-7$. Now we assume the claim hold for any $r<-7$. Again from the recurrence relation we have,
\begin{equation*}
    \begin{aligned}
    a_{r,p}=&-7a_{r+1,p}-3\cdot7a_{r+2,p}-7^2a_{r+3,p}-3\cdot 7^2a_{r+4,p}-7a_{r+4,p+1}\\
    &-7^3a_{r+5,p}
    -5\cdot 7a_{r+5,p+1}
    -7^3a_{r+6,p}-7^2a_{r+6,p+1}+a_{r+7,p+2}.
    \end{aligned}
\end{equation*}

Now by assumption we have that the 
7-adic valuation of the left hand side satisfies required inequality. 
\begin{equation*}
\begin{aligned}
    \pi_7(a_{r,p})\geq&\min\Bigg\{\pi_7(a_{r+1,p})+1, \pi_7(a_{r+2,p})+1, \pi_7(a_{r+3,p})+2
    ,\pi_7(a_{r+4,p})+2,\\
    &\pi_7(a_{r+4,p+1})+1, \pi_7(a_{r+5,p})+3, \pi_7(a_{r+5,p+1})+1,  \pi_7(a_{r+6,p})+3,\\
    &\pi_7(a_{r+6,p+1})+2, \pi_7(a_{r+7,p+2})\Bigg\}\\
    =&\left\lfloor\frac{7p-2r+3}{4}\right\rfloor.
\end{aligned}
\end{equation*}
\end{proof}
Notice that we have used $a_{r,p}$ to denote $a_{r,p}^7$.

\begin{lemma}\label{L.3.5}
Let $r$ be a non zero integer. Then $S_{r,13}(z)$ is a Laurent polynomial in $g_{13}(z)$ of the form  $S_{r,13}(z)=\sum_{p=-\infty}^{\infty}a_{r,p}^{13}g_{13}^p(z)$,
where $a_{r,p}$ is an integer divisible by $13$,
\begin{equation}\label{E:a13}
\pi_{13}(a_{r,p}^{13})\geq \left\lfloor\frac{13p-7r+13}{14}\right\rfloor,
\end{equation}
for $r>0$, $a_{r,p}\not=0$ for $\floor{\frac{7r+12}{13}}\leq p \leq 7r$. For $r<0$, $a_{r,p}\not=0$ for $\floor{\frac{7r+12}{13}} \leq p\leq 0$.
\end{lemma}
\begin{proof}
For $r>0$, this is proved in \cite{AO}. For $r<0$, we modify \eqref{E:13} by dividing it by $\phi_{13}^{13}(z)$ and $g_{13}^7(z)$. Then we derive $S_{r,13}(z)$ for negative $r$ using theorem \ref{T:3.1}.
\begin{equation}\label{E:b13}
S_{r,13}(z):=\sum_{\floor{\frac{7r+12}{13}}\leq p\leq 0}a_{r,p}^{13}g_{13}^{p}(z).
\end{equation}
Here $\{a_{r,p}^{13}\}_{r<0, p \leq 0}$ given in the matrix below, \\

$\left[\begin{smallmatrix}
\cdots&& & & & & &  13 \\
\cdots&& & & & & -2\cdot 13&  -13 \\
\cdots&& & & & & & 13^2\\
\cdots&& & & & 26 & 12\cdot 13^3 & 467\cdot 13^3\\
\cdots&& & & & -260 & 2018\cdot 13^2& 467\cdot 13^3\\
\cdots&& & & 130 & -20\cdot 13^4 & 36\cdot 13^4 & 2807\cdot 13^3 \\
\cdots&& & & 98\cdot 13& -26336\cdot 13^2 & -12\cdot 13^5 & 935\cdot 13^4\\
\cdots&&& -70\cdot 13& 84\cdot 13^4 & -684\cdot 13^4 & -120\cdot 13^5&3743\cdot 13^4\\
\cdots& && - 162\cdot 13&  176544\cdot13^2&  - 15996\cdot13^3 &- 9030\cdot13^4& 937.13^5\\
\cdots&&  238\cdot 13^5 & -396\cdot 13^4 & 4740\cdot 13^4 & -192\cdot 13^5 & -24\cdot 13^7 & 1871\cdot 13^5\\
\cdots&& -902\cdot 13 & -737836\cdot 13^2 & 210588\cdot 13^3 & 10722 \cdot 13^4 & -17806\cdot 13^5 & -13^6\\
\cdots &-418\cdot 13& 1260\cdot 13^4 & -16812\cdot 13^4 & 7416\cdot 13^5 & 120\cdot 13^7 & -336\cdot 13^7 & -1873\cdot 13^6\\
 \cdots13& 51\cdot 13^3& 125764\cdot 13^3 &-77470\cdot 13^4& 28214\cdot 13^5 &10000\cdot 13^6 & -815\cdot 13^7 & -72\cdot 13^8\\
 \cdots&\cdots&\cdots&\cdots&\cdots&\cdots&\cdots&\cdots\\
\end{smallmatrix}\right]$\\

From the above matrix we can see that the result holds for $r=-13 $ to $-1$. Now fix $r<-13$. Assume the result hold for all negative numbers greater than $r$. Using theorem \ref{T:3.1}  and \eqref{E:13} we see that $S_{r}$ satisfies the following recursive formula when $r\leq -14$.
\begin{equation}\label{E:13.r}
\begin{aligned}
     S_{r,13}(z):=&\sum_{i=1}^{13}(-1)^{i+1}C_{-i}S_{r+i,13}(z),\\
   \text{where}\;\; C_{-i}:=&(-1)^{i+1}\sum_{\rho=\ceil{\frac{14-i}{2}}}^{7}m_{13-i,\rho}g_{13}^{\rho-7}(z),\;\;\; 1\leq i \leq 13.
\end{aligned}
\end{equation}
Notice that $m_{0,7}:=1$ here. Now combining \eqref{E:b13} and \eqref{E:13.r}, for $r\leq -14$ we have,
\begin{equation}\label{E:c13}
a_{r,p}^{13}=\sum_{i=1}^{13}\sum_{t=\floor{\frac{7r+7i+12}{13}}}^0m_{\footnotesize{13-i,p-t+7}}\cdot a_{\footnotesize{{r+i,t}}}^{13}.   
\end{equation}

Now using \eqref{E:13} we see that,
\[\pi_{13}\left(m_{\footnotesize{r,p}}\right)\geq\left\lfloor\frac{13p-7r+13}{14}\right\rfloor.\]

Then using the induction hypothesis we have,

\begin{equation*}
    \begin{aligned}
        \pi_{13}(a_{r,p}^{13})&=\min_{1\leq i\leq 13,\; \floor{\frac{7r+7i+12}{13}}\leq t\leq 0}\left\{\pi_{13}(m_{13-i,p-t+7})+\pi_{13}(a_{r+i,t}^{13})\right\}.\\
\geq&\left\{\left\lfloor\frac{13(p-t+7)-7(13-i)+13}{14}\right\rfloor+ \left\lfloor\frac{13(t)-7(r+i)+13}{14}\right\rfloor\right\},\\
=&\left\{\left\lfloor\frac{13p-13t+7i+13}{14}\right\rfloor+ \left\lfloor\frac{13t-7r-7i+13}{14}\right\rfloor\right\},\\
\geq&\left\lfloor\frac{13p-7r+13}{14}\right\rfloor,\;\;\;\text{for all}\;\;i, t.
    \end{aligned}
\end{equation*}

Here we used the fact that if $X,Y$ are integers then
\[\left\lfloor\frac{X}{14}\right\rfloor+\left\lfloor\frac{Y}{14}\right\rfloor\geq \left\lfloor\frac{X+Y-13}{14}\right\rfloor.\]
\end{proof}

\begin{lemma}\label{L.3.6}
Let $r$ be a non zero integer.
\begin{enumerate}
    \item $\pi_5\left(S_{5,r}\right)=\pi_{5}\left(a_{r,\floor{\frac{r+4}{5}}}^5\right)=1$ iff $r\not\equiv 1,2\pmod{5}.$
    \item $\pi_7\left(S_{7,r}\right)=\pi_{7}\left(a_{r,\floor{\frac{2r+6}{7}}}^7\right)=1$ iff $r\not\equiv 1,4\pmod{7}.$
    \item $\pi_{13}\left(S_{13,r}\right)=\pi_{13}\left(a_{r,\floor{\frac{7r+12}{13}}}^{13}\right)=1$ iff $r\not\equiv 10\pmod{13}.$
\end{enumerate}
For all the other cases $\ell-$adic order of $S_{\ell,r}$ is greater than or equal to 2.
\end{lemma}

\begin{proof}
The proof follows from induction on $r$. To prove the induction step, we use the recursive expression for $S_{\ell,r}$ that can be obtained from the modular equations. We demonstrate this by proving $(1)$ when $r$ is a positive multiple of $5$.\\

By lemma $4$ chapter $8$ in \cite{K}, we have $S_{5,5}$ only divisible by $5$ and $\pi_5(a_{5,1}^5)=1$. Now we assume for $r>1,  \pi_5(a_{5r-5,r-1}^5)=1$. Then by the recursive formula for $S_{5,r}$ we have 
\[S_{5,5r}=5^2g_5S_{5,5r-1}+5^2g_5S_{5r-2}+15g_5S_{5,5r-3}+5g_5S_{5,5r-4}+g_5S_{5,5r-5}.\]
Then the result follows comparing the coefficient of $g_5^r.$

\end{proof}

\begin{remark}\label{R.3.7}
For $\ell=5,7,$ and $13$, the $\ell-$adic order of $S_{\ell,r}$ is equal to the $\ell-$adic order of the coefficient of $g_{\ell}(z)$ with the least power when $r$ is positive.   Otherwise it is the $\ell-$adic order of the coefficient of $g_{\ell}(z)$ with the largest power. This follows from lemma \ref{L.3.6} and inequalities \eqref{E:a5} \eqref{E:a7}, and \eqref{E:a13}.
\end{remark}
\vspace{3mm}

Let $V_{\ell}$ be the vector space of modular functions on $\Gamma_0(\ell)$ where $\ell=2,3,5,7,11,13$ and $17$, which are holomorphic everywhere except possibly at 0 and $\infty$. $V$ is mapped to itself by the linear transformation,
\[T_{\lambda}:f(\tau)\rightarrow U_{\ell}\left(\phi_{\ell}(\tau)^{\lambda}f(\tau)\right),\]
where $\lambda$ is an integer. Let $(C_{\mu,\nu}^{\lambda})_{\mu,\nu}$ be the matrix of the linear transformation $T_\lambda$ with respect to a triangular basis of $V_{\ell}$ .\\

For primes $\ell=2,3,5,7$ and $13$, we can take \{$g_{\ell}^{\nu}(\tau)|\nu\in\mathbb{Z}^+$\} as an upper triangular basis when $\lambda>0$. For $\lambda<0$ we can take  \{$g_{\ell}^{\nu}(\tau)|\nu\in\mathbb{Z}^-$\} as a triangular basis. For $\ell=11$ and $17$, finding such basis is complicated and they were derived by O.L Atkin \cite{A2} and Kim Hughes \cite{H} respectively.\\

Now let \{$J_{\ell,\nu}(\tau)|\nu\in\mathbb{Z}$\} 
be an upper triangular basis for $V_{\ell}$. Then we have
\begin{equation}\label{E:C}
U\left(\phi_{\ell}(\tau)^{\lambda}J_{\ell,\mu}\right)=\sum_{\nu}C_{\mu,\nu}^{\lambda}J_{\ell,\nu}.
\end{equation}
Therefore, the Fourier series of $T_\lambda(J_{\ell,\nu})$ has all coefficients divisible by $\ell$ if and only if,

$$C_{\mu,\nu}^{\lambda}\equiv0\pmod{\ell} \ \mbox{for all $\nu$}.$$

Now we define $\theta_{\ell}(\lambda,\mu)=1$ if all the coefficients of $U(\phi_{\ell}^{\lambda}J_{\ell,\nu})$ divisible by $\ell$. Otherwise we put $\theta_{\ell}(\lambda,\mu)=0$.  \\

\begin{lemma}\label{L:3.5}
For $\ell=5$,
\begin{equation}\label{E:3.5}
\begin{aligned}
\theta_{5}(\lambda,\mu)&=\theta_{5}(\lambda+5,\mu),\\ \theta_{5}(\lambda,\mu+1)&=\theta_{5}(\lambda+6,\mu).
\end{aligned}
\end{equation}

\begin{table}[htp]
\begin{tabular}{|l@{\hspace{16pt}}|*{6}{c|}}
\hline
& \multicolumn{5}{c|}{$\lambda$} \\
\cline{1-6}
\hline
$\mu$& 0 & 1 & 2 & 3 & 4 \\
\hline\hline
\cline{1-6}
0 & 0 & 1 & 1 & 0 & 0 \\
\hline
\end{tabular}\\
\vspace{2mm}
\caption{Values of $\theta_5(\lambda,\mu)$.}\label{t:3.5}
\end{table}
\end{lemma}

\begin{proof}

First notice that using \eqref{E:5} and \eqref{E:2.2} we have that

\[U\left(\phi_5^{\lambda+5}(\tau)g_5^{\mu}(\tau)\right)=g_5(\tau)U\left(\phi_5^{\lambda}(\tau)g_5^{\mu}(\tau)\right)\pmod{5}.\]
Now using \eqref{E:C} we have,
\[C_{\mu,\nu}^{\lambda+5}\equiv C_{\mu,\nu-1}^{\lambda}\pmod{5}.\]
Hence we have the first equality of \eqref{E:3.5}. To get the second equality, consider that,
\[U\left(\phi_5^{\lambda}(\tau)g_{5}^{\mu+1}(\tau)\right)=g_5^{-1}(\tau)U\left(\phi_5^{\lambda+6}(\tau)g_5^{\mu}(\tau)\right).\]
Hence we have
\[C_{\mu+1,\nu}^{\lambda}=C_{\mu,\nu+1}^{\lambda+6}.\]
Therefore we have the second equality of \eqref{E:3.5}. Now using \eqref{E:a5}, we can find $\theta_5(\lambda,\mu)$ values for all $\lambda$ and $\mu$.
\end{proof}

\begin{lemma}\label{L:3.6}
For $\ell=7,$
\begin{equation}\label{E:3.7}
    \begin{aligned}
    \theta_7(\lambda,\mu)&=\theta_7(\lambda-7,\mu),\\ \theta_7(\lambda,\mu+1)&=\theta_7(\lambda+4,\mu).
    \end{aligned}
\end{equation}

    \begin{table}[htp]
    \begin{tabular}{|l@{\hspace{16pt}}|*{8}{c|}}
\hline
& \multicolumn{7}{c|}{$\lambda$} \\
\cline{1-8}
\hline
$\mu$& 0 & 1 & 2 & 3 & 4 & 5 & 6\\
\hline\hline
\cline{1-6}
0 & 0 & 1 & 0 & 0 & 1 & 0 & 0 \\
\hline
\end{tabular}\\
\vspace{2mm}
    \caption{Values of $\theta_{7}(\lambda,\mu)$.}\label{t:3.6}
    \end{table}

First notice that using \eqref{E:7} and \eqref{E:2.2}, 
\[U\left(\phi_7^{\lambda+7}(\tau)g_7^{\mu}(\tau)\right)=g_7^2(\tau)U\left(\phi_7^{\lambda}(\tau)g_7^{\mu}(\tau)\right)\pmod{7}.\]
Now using \eqref{E:C} we have that,
\[C_{\mu,\nu}^{\lambda+7}\equiv C_{\mu,\nu-2}^{\lambda}\pmod{7}.\]
Hence we have the first equality of \eqref{E:3.7}. To get the second equality, consider that,
\[U\left(\phi_7^{\lambda}(\tau)g_{7}^{\mu+1}(\tau)\right)=g_7^{-1}(\tau)U\left(\phi_7^{\lambda+4}(\tau)g_7^{\mu}(\tau)\right).\]
Hence we have
\[C_{\mu+1,\nu}^{\lambda}=C_{\mu,\nu+1}^{\lambda+4}.\]
Therefore we have the second equality of \eqref{E:3.7}. Now using \eqref{E:a7}, we can find $\theta_7(\lambda,\mu)$ values for all $\lambda$ and $\mu$.
\end{lemma}

\begin{lemma}\label{L:3.7}
For $\ell=13,$

\begin{equation}\label{E:3.13}
\begin{aligned}
\theta_{13}(\lambda,\mu)&=\theta_{13}(\lambda-13,\mu),\\ \theta_{13}(\lambda,\mu+1)&=\theta_{13}(\lambda+2,\mu).
\end{aligned}
\end{equation}

\begin{center}
    \begin{table}[htp]
    \begin{tabular}{|l@{\hspace{12pt}}|*{13}{c|}}
\hline
& \multicolumn{13}{c|}{$\lambda$} \\
\cline{1-13}
\hline
$\mu$& 0 & 1 & 2 & 3 & 4 & 5 & 6 & 7 & 8 & 9 & 10 & 11 & 12\\
\hline\hline
\cline{1-6}
0 & 0 & 0 & 0 & 0 & 0 & 0 & 0 & 0 & 0 & 0 & 1 & 0 & 0\\
\hline
\end{tabular}\\
\vspace{2mm}
    \caption{Values of $\theta_{13}(\lambda,\mu)$.}\label{t:3.7}
    \end{table}
    \end{center}
\end{lemma}

\begin{proof}
Again notice that using \eqref{E:13} and \eqref{E:C} we have that
\[U\left(\phi_{13}^{\lambda+13}(\tau)g_{13}^{\mu}(\tau)\right)=g_{13}^7(\tau)U\left(\phi_{13}^{\lambda}(\tau)g_{13}^{\mu}(\tau)\right)\pmod{13}.\]
This implies that
\[C_{\mu,\nu}^{\lambda+13}\equiv C_{\mu,\nu-7}^{\lambda}\pmod{13}.\]
Hence we have the first equality of \eqref{E:3.13}. To get the second equality, consider that,
\[U\left(\phi_{13}^{\lambda}(\tau)g_{13}^{\mu+1}(\tau)\right)=g_{13}^{-1}(\tau)U\left(\phi_{13}^{\lambda+13}(\tau)g_{13}^{\mu}(\tau)\right).\]
Hence we have
\[C_{\mu+1,\nu}^{\lambda}=C_{\mu,\nu+1}^{\lambda+2}.\]
Therefore we have the second equality of \eqref{E:3.13}. Now using \eqref{E:a13}, we can find $\theta_{13}(\lambda,\mu)$ values for all $\lambda$ and $\mu$.
\end{proof}

\begin{lemma}\label{L:3.9}
For $\ell=17,$
\begin{equation}\label{E:3.17}
\begin{aligned}
\theta_{17}(\lambda,\mu)&=\theta_{17}(\lambda-17,\mu),\\ \theta_{17}(\lambda,\mu)&=\theta_{17}(\lambda+6,\mu-4).
\end{aligned}
\end{equation}

\begin{center}
    \begin{table}[htp]
    \begin{tabular}{|l@{\hspace{9pt}}|*{17}{c|}}
\hline
& \multicolumn{17}{c|}{$\lambda$} \\
\cline{1-17}
\hline
$\mu$& 0 & 1 & 2 & 3 & 4 & 5 & 6 & 7 & 8 & 9 & 10 & 11 & 12 & 13 & 14 & 15 & 16\\
\hline\hline
\cline{1-18}
0 & 0 & 0 & 0 & 1 & 0 & 0 & 0 & 0 & 0 & 0 & 0 & 0 & 0 & 0 & 0 & 0 & 0\\
\hline
1 & 0 & 0 & 0 & 1 & 0 & 0 & 0 & 0 & 0 & 0 & 0 & 0 & 0 & 0 & 0 & 0 & 0\\
\hline
2 & 0 & 0 & 0 & 1 & 0 & 0 & 0 & 0 & 0 & 0 & 0 & 0 & 0 & 0 & 0 & 0 & 0\\
\hline
3 & 0 & 0 & 0 & 1 & 0 & 0 & 0 & 0 & 0 & 0 & 0 & 0 & 0 & 0 & 1 & 0 & 0\\
\hline
\end{tabular}\\
\vspace{2mm}
    \caption{Values of $\theta_{17}(\lambda,\mu)$.}\label{t:3.9}
    \end{table}
    \end{center}
\end{lemma}
\begin{proof}
See \cite[p. 518]{H} for a proof. 
\end{proof}

\section{Constructing Modular functions}\label{CMF}

We construct a sequence of modular functions that are the generating functions for the $p_{[1^c{\ell}^d]}(n)$ restricted to certain arithmetic progressions. This generalizes Gordon's construction for "$k-$color"  partitions. Here we use \eqref{E:2.2} repeatedly.

 Let $L_{0,\ell} := 1$ and $\delta_{\ell}:=\frac{\ell^2-1}{24}$,
\begin{align*}
L_{1,\ell}(\tau):&=U_{\ell}\left(\phi_{\ell}(\tau)^c\prod_{n=1}^{\infty}\frac{(1-q^{\ell\cdot n})^d}{(1-q^{\ell\cdot n})^d}\right),\\
&=U_{\ell}\left(q^{\delta_{\ell}\cdot c}\prod_{n=1}^{\infty}\frac{(1-q^{\ell^2\cdot n})^c(1-q^{\ell\cdot n})^d}{(1-q^n)^c(1-q^{\ell\cdot n})^d}\right),\\
&=\prod_{n=1}^{\infty}(1-q^{\ell\cdot n})^c(1-q^n)^d\sum_{m\geq \lceil\frac{\delta_{\ell}\cdot c}{\ell}\rceil}^{\infty}p_{[1^c\ell^d]}(\ell\cdot m-\delta_{\ell}\cdot c)q^m.
\end{align*}

    Similarly we define,
    \begin{align*}
        L_{2,\ell}(\tau) :&=U_{\ell}\left(\phi_{\ell}^d(\tau)L_1(\tau)\right),\\
    L_{2,\ell}(\tau)& =\prod_{n=1}^{\infty}(1-q^{\ell\cdot  n})^d(1-q^n)^c\sum_{m\geq \left\lceil\frac{\delta_{\ell}\cdot d+\lceil\frac{\delta_{\ell}\cdot c}{\ell}\rceil }{\ell}\right\rceil}^{\infty}p_{[1^c\ell^d]}(\ell^2m-\delta_{\ell}\cdot \ell\cdot d -\delta_{\ell}\cdot c)q^m.
    \end{align*}

    Now, to get an equation for higher powers, we define,
    \begin{equation}\label{E:2.8}
    L_{r,\ell}:=U_{\ell}\left(\phi_{\ell}^{\lambda_{r-1}}(\tau)L_{r-1,\ell}\right),
    \end{equation}
    where
    \begin{equation}\label{E:lr}
        \quad \lambda_{r}=
\left\{
	\begin{array}{ll}
		 c & \mbox{if  $r$ is even }, \\
		d & \mbox{if  $r$ is odd}.
	\end{array}
\right.
\end{equation} \\
	
	Then by a short calculation using \eqref{E:2.2}, there exists integers  $n_r$ and $\mu_r$ such that,
	
\begin{align}
\label{E:L_r}
    L_{2r,\ell}(\tau)& =\prod_{n=1}^{\infty}(1-q^n)^c(1-q^{\ell\cdot n})^d\sum_{m\geq \mu_{2r}}p_{[1^c\ell^d]}(\ell^{2r}m+n_{2r})q^m, \\
    \notag
    L_{2r-1,\ell}(\tau)& =\prod_{n=1}^{\infty}(1-q^{\ell\cdot  n})^c(1-q^n)^d\sum_{m\geq \mu_{2r-1}}p_{[1^c\ell^d]}(\ell^{2r-1}m+n_{2r-1})q^m.
\end{align}
\vspace{3mm}

  From \eqref{E:2.8} and \eqref{E:L_r} we can see that,
\[n_{2r,\ell}(c,d)=-\delta_{\ell}\cdot d\cdot \ell^{2r-1}+n_{2r-1},\]
\[n_{2r-1,\ell}(c,d)=-\delta_{\ell}\cdot c\cdot \ell^{2r-2}+n_{2r-2}.\]
Since $n_{0,\ell}=0$, using above recurrence relations we have,
\vspace{2mm}
\begin{itemize}[label={}]
    \item $n_{1,\ell}(c,d)=-\delta_{\ell}\cdot c$
    \item $n_{2,\ell}(c,d)=-\delta_{\ell}\cdot \ell\cdot d -\delta_{\ell}\cdot c$
\end{itemize}
Using the summation of a geometric series,
\begin{align}
\label{E:n_r}
n_{2r-1,\ell}(c,d)&=-c\left(\dfrac{\ell^{2r}-1}{24}\right)-\ell\cdot d\left(\dfrac{\ell^{2r-2}-1}{24}\right).\\
\notag
n_{2r,\ell}(c,d)&=-c\left(\dfrac{\ell^{2r}-1}{24}\right)-\ell\cdot d\left(\dfrac{\ell^{2r}-1}{24}\right) .
\end{align}
From this we have that,
    \[24 \cdot n_{2r-1,\ell}\equiv (c+\ell\cdot d) \mod \ell^{2r-1}  \text{  and  }
    24\cdot n_{2r,\ell}\equiv (c+\ell\cdot d) \mod \ell^{2r}.\] 

    Therefore, each $n_r$ satisfies,

    \[24n_{r,\ell}(c,d)\equiv (c+\ell\cdot d) \mod \ell^r.\]

    Now, we need to find $\mu_r$ in terms of integers $c,d$. Using \eqref{E:2.8}, we have that,

    \begin{equation}
    \label{E:murR}
    \textcolor{black}{\mu_{r,\ell}=\left\lceil \frac{\delta_{\ell}\lambda_{r-1}+\mu_{r-1}}{\ell}\right\rceil.}
    \end{equation}
\vspace{2mm}

    Notice also that $\mu_r$ is the least integer $m$ such that ${\ell}^r m+n_r\geq 0$, which implies that,

    \begin{align}
    \label{E:mur}
    \mu_{2r-1,\ell}& =\left\lceil\frac{\ell\cdot c+d}{24}-\dfrac{c+\ell\cdot d}{24\cdot\ell^{2r-1}}\right\rceil,\\
    \notag
    \mu_{2r,\ell}& =\left\lceil\frac{c+\ell\cdot d}{24}-\dfrac{c+\ell\cdot d}{24\cdot{\ell}^{2r}}\right\rceil.
   \end{align}

Following Gordon, we represent these formulas in the following form.
\begin{align}
\label{E:2.11}
\mu_{2r-1,\ell}& =\left\lceil{\frac{\ell\cdot  c+d}{24}}\right\rceil+\omega_{\ell}(c,d) \quad  \mbox{if $|c+\ell\cdot d|<{\ell}^{2r-1}$},\\
\notag
\mu_{2r,\ell}& =\left\lceil{\frac{c+\ell\cdot d}{24}}\right\rceil+\omega_{\ell}(c,d) \quad \mbox{if $|c+\ell\cdot d|< {\ell}^{2r}$ },\\
\notag
\text{where}\;\;\omega_{\ell}(c,d)&=
\left\{
	\begin{array}{ll}
		 1 & \mbox{if  $c+ \ell\cdot d<0$ and $24|(c+\ell\cdot d)$}, \\
		 0 & \mbox{Otherwise}. \\
	\end{array} \right. 
\end{align}

\section{Proofs of congruences}\label{P}
\vspace{2mm}

Now we define,
\begin{equation}\label{E:2.7}
A_{r,\ell}(c,d):=\sum_{i=0}^{r-1}\theta_{\ell}(\lambda_{i},\mu_i),
\end{equation}
for any positive integer $r$ and integers $c,d$. We put $A_0:=0$.\\
If $$f(\tau):=\sum_{n\geq n_0}a(n)q^n,$$ we define,
\begin{equation}\label{E:5.0}
\pi\left(f(\tau)\right):=\mbox{min}_{n\geq n_0}\left\{\pi(a(n)\right\}.
\end{equation}

We prove $\pi(L_r)\geq A_r(c,d)$.

\begin{proof}
For $\ell=5,7$ and $13$, we calculate,
\begin{equation*}\label{E:5.1}
    \begin{aligned}
    L_{0,\ell}:=&1,\\
    L_{1,\ell}:=&U_{\ell}\left(\phi_{\ell}^{\lambda_0}g_{\ell}^{\mu_0}\right)=\sum_{\nu_1=\mu_1}C_{\mu_0,\nu_1}^{\lambda_0}g_{\ell}^{\nu_1},\;\;\; \text{Here we have $\pi(L_{1,\ell})\geq \theta_{\ell}(\lambda_0,\mu_0)$},\\
    L_{2,\ell}:=&U_{\ell}\left(\phi_{\ell}^{\lambda_1}L_1\right)=U_{\ell}\left(\phi_{\ell}^{\lambda_1}\sum_{\nu_1=\mu_1}C_{\mu_0,\nu_1}^{\lambda_0}g_{\ell}^{\nu_1}\right)=\sum_{\nu_1=\mu_1}C_{\mu_0,\nu_1}^{\lambda_0}U_{\ell}\left(\phi^{\lambda_1}g_{\ell}^{\nu_1}\right),\\
    =&\sum_{\nu_1=\mu_1}C_{\mu_0,\nu_1}^{\lambda_0}\sum_{\nu_2=\mu_2}C_{\nu_1,\nu_2}^{\lambda_1}g_{\ell}^{\nu_2}= \sum_{\nu_2=\mu_2}C_{\mu_0,\mu_1}^{\lambda_0}C_{\mu_1,\nu_2}^{\lambda_1}*g_{\ell}^{\nu_2}.\\
    &\text{Here we have $\pi(L_{2,\ell})\geq \theta_{\ell}(\lambda_0,\mu_0)+\theta_{\ell}(\lambda_1,\mu_1)$}.\\
    \end{aligned}
\end{equation*}

Now by induction we have that,
\begin{equation}
    L_{r,\ell}:=\sum_{\nu_r=\mu_r}C_{\mu_0,\mu_1}^{\lambda_0}C_{\mu_1,\mu_2}^{\lambda_1}\cdots C_{\mu_{r-1},\nu_r}^{\lambda_{r-1}}*g_{\ell}^{\nu_r}.
\end{equation}
Here $\lambda_r$ and $\mu_r$ are defined in the previous section. Also notice that it is sufficient to consider the first coefficient of the series expansion of $S_{\ell,r}$ to get a lower bound for $\pi(L_{r,\ell})$ by remark \ref{R.3.7}.\\
\end{proof}

\begin{proof}[Proof of corollary \ref{C:1.3}]
Recall \eqref{E:2.7},\\
\[A_{r,\ell}(c,d)=\sum_{i=0}^{r-1}\theta_{\ell}(\lambda_i,\mu_i).\]

In \cite{P} we used Gordon's argument from Section 4 of \cite{GB} to obtained the result for $\ell=11$. Here we use similar argument for primes $\ell=5,7,13$ with $k$ replaced by $c + \ell d$, . Note that Gordon's calculations had terms involving $11k$, which will be written in a more symmetric shape here using the fact that $\ell(c+\ell d) \equiv \ell c + d \pmod{24}$ for any prime $\ell \neq 2,3$.

\begin{align*}
A_{r,\ell}(c,d) & =\sum_{i=0}^{\log_{\ell}(c+\ell d)}\theta_{\ell}(\lambda_i,\mu_i)+\sum_{i=\log_{\ell}(c+\ell d)}^{r-1}\theta_{\ell}(\lambda_i,\mu_i) \\
&=\sum_{i=0}^{\log_{\ell}(c+\ell d)}\theta_{\ell}(\lambda_i,\mu_i)+N_{1,\ell}\cdot\theta_{\ell}\left(d,\ceil[\Big]{\dfrac{\ell c+d}{24}}+\omega_{\ell}(c,d)\right) \\
& \qquad \qquad \qquad \qquad +N_{2,\ell}\cdot\theta_{\ell}\left(c,\ceil[\Big]{\dfrac{c+\ell d}{24}}+\omega_{\ell}(c,d)\right).
\end{align*}

Here $N_{1,\ell}$ is the number of odd integers and $N_{2,\ell}$ is the number of even integers in the interval $\bigg(\log_{\ell}|c+\ell d|,r-1\bigg)$ respectively.\\

\begin{equation}\label{E:5.44}
\mbox{Set} \quad \alpha_{\ell}:=\alpha_{\ell}(c,d)=\theta_{\ell}\left(d,\ceil[\Big]{\dfrac{\ell c+d}{24}}+\omega_{\ell}(c,d)\right)+\theta_{\ell}\left(c,\ceil[\Big]{\dfrac{c+\ell d}{24}}+\omega_{\ell}(c,d)\right)
.\end{equation}

Now if $r\leq \log_{\ell}|c+\ell d|+1$ then $N_{1,\ell}=N_{2,\ell}=0$,
\[A_r\leq \log_{\ell}|c+\ell d|.\]

If $r>\log_{\ell}|c+\ell d|+1$,
\[\left|N_{1,\ell}-\frac{1}{2}(r-1-\log_{\ell}(c+\ell d))\right|+\left|N_{2,\ell}-\frac{1}{2}(r-1-\log_{\ell}(c+\ell d))\right|<1.\]
\vspace{2mm}
Now consider,
\[\left|A_{r,\ell}-\frac{1}{2}\alpha_{\ell}(r-1-\log_{\ell}(c+\ell d))\right|<2+\log_{\ell}|c+\ell d|,\]
\[\left|A_{r,\ell}-\frac{\alpha_{\ell} r}{2} \right|<2+\frac{\alpha_{\ell}}{2}+(1+\frac{\alpha_{\ell}}{2})\log_{\ell}|c+\ell d|.\]\\
So we have $A_{r,\ell}=\frac{1}{2}\alpha_{\ell} r+\mathcal{O}\left(\log{}|c+\ell d|\right)$.
\vspace{2mm}\\
Now we  prove the condition for $\alpha_{\ell}$. As in \cite{P}, the proof is complete once we show that $\alpha_{\ell}$ only depends on $c+\ell d$ with the period $24$ when $\ell=5,7$ and $13$. Periodicity follows by lemmas \ref{L:3.5}, \ref{L:3.6}, \ref{L:3.7}, and  $\alpha_{\ell}(c+\ell d)$ is invariant under the maps, 
\[c\rightarrow{c+24-\ell k} \quad\text{and}\quad d\rightarrow{d+k}\quad \mbox{for each integer k}.\]

When $c+\ell d<0$ and $24|(c+\ell d)$,  using \eqref{E:2.11}, $\omega_{\ell}(c,d)$ is $1.$ Thus, using \eqref{E:5.44} we have to change the last column of the table \ref{t:1.3} for entries when $\ell=5$ and $7$.

Also note that when $\ell=17$, $\alpha_{\ell}(c+\ell d)$ is invariant under the following maps by lemma \ref{L:3.9}, 
\[c\rightarrow{c+96-\ell k} \quad\mbox{and}\quad d\rightarrow{d+k}\quad \mbox{for each integer k}.\]
\end{proof}

\begin{remark}
We do not provide a proof for $\ell=17$ here, since it is similar to the proof of theorem $3$ in \cite{H}. In fact, the proof of theorem $1.1$ in \cite{P} and the proof of theorem $2$ in \cite{GB} are similar.
\end{remark}

\begin{proof}[Proof of corollary \ref{C1.5}]
First we assume $A_r(c,d)=0$ for all $r$ and let $\ell=5,7,$ and $13$. Now by \eqref{E:2.7}, $\theta_{\ell}(\lambda_r,\mu_r)=0$ and hence $\pi_{\ell}(U_{\ell}(\phi_{\ell}^{\lambda_r}g_{\ell}^{\mu_r}))=0.$
\begin{equation*}
    L_{r,\ell}=C_{\mu_0,\mu_1}^{\lambda_0}C_{\mu_1,\mu_2}^{\lambda_1}\cdots C_{\mu_{r-1},\mu_r}^{\lambda_{r-1}}g_{\ell}^{\mu_r}+ \cdots.
\end{equation*}
 Since all $C_{\mu_j,\mu_{j+1}}^{\lambda_j}$
are not divisible by $\ell$ for all $j$ from lemma \ref{L.3.6}, we have $\pi_{\ell}\left(L_{r,\ell}\right)=0$.\\

For $\ell=11,17$ we only prove it for $r=1$. Suppose
$A_{1,\ell}(c,d)=0$. Then $\theta_{\ell}(\lambda_0,\mu_0)=0$. Now consider,
\[L_{1,\ell}=\sum_{\nu_1=\mu_1}C_{\mu_0,\nu_1}^{\lambda_0}J_{\ell, \nu_1}.\]
 Since $\theta_{\ell}(\lambda_0,\mu_0)=0$, there is some $s_{\mu}\geq \mu_0$ such that $C_{\mu_0,s_{\mu}}^{\lambda_0}\not \equiv 0\pmod{\ell}$. Hence we have $L_{1,\ell}\not\equiv0\pmod{\ell}.$
 Here $J_{\ell,\nu}$ are the bases used by Gordon in \cite{GB} and by Hughes in \cite{H} for $\ell=11,$ and $17$ respectively.
\end{proof}

\vspace{5mm}
\begin{proof}[Proof of corollary \ref{C:1.7}] 
We first state the result of Chan, Wang, and Yang about representations of $C\Phi_k(q)$ for $k=5, 7, 11,$ and $13$.
\begin{theorem}[H.H. Chan, L. Wang, Y. Yang, 2018 \cite{CWY}]\label{T:1.3}
    Let $C\Phi_k(q)$ be the $k$-Colored generalized Frobenius Partitions,
    \begin{equation}\label{E:1.2}
    C\Phi_5(q)=\prod_{n=1}^{\infty}\frac{1}{1-q^{5n}}+25 q\prod_{n=1}^{\infty}\frac{(1-q^{5n})^5}{(1-q^n)^6}.
    \end{equation}
    \begin{equation}\label{E:1.3}
    C\Phi_7(q)=\prod_{n=1}^{\infty}\frac{1}{1-q^{7 n}}+49\cdot q\prod_{n=1}^{\infty}\frac{(1-q^{7 n})^3}{(1-q^n)^4}+343\cdot q^2\prod_{n=1}^{\infty}\frac{(1-q^{7 n})^7}{(1-q^n)^8}.
    \end{equation}

     \begin{equation}\label{E:1.4}
    C\Phi_{11}(q)=\prod_{n=1}^{\infty}\frac{1}{1-q^{11 n}}+11\sum_{j=1}^{\infty}p(11j-5)q^j.
    \end{equation}

    \begin{equation}\label{E:1.5}
        C\Phi_{13}(q)= \frac{1}{(q^{13};q^{13})_{\infty}}+13\sum_{n=0}^{\infty}p(13n-7)q^n+26\cdot q\prod_{n=1}^{\infty}\frac{(1-q^{13n})}{(1-q^n)^2}.
        \end{equation}

    \end{theorem}

We combine theorem \ref{T:1.1} and theorem \ref{T:1.3} to prove congruences of $k-$ colored generalized Frobenius partitions.
First notice that using \eqref{E:1.2} we have that
\begin{equation}\label{E:5.4}
    C\Phi_5(q)=\sum_{n=0}^{\infty}\bigg({p_{[1^05^1]}(n)+ 25p_{[1^65^{-5}]}(n-1)}\bigg)q^n.
\end{equation}
Now using \eqref{E:n_r} we have that, 
\begin{equation}\label{E:5.5}
    n_{2r,5}(0,1)=\frac{5-5^{2r+1}}{24},\;\;\;n_{2r,5}(6,-5)=\frac{19\cdot 5^{2r}-19}{24}.
\end{equation}
Observe that using \eqref{E:lr} and \eqref{E:mur} we have that
\begin{equation}\label{E:5.6}
\mu_{r,5}(0,1)
=\begin{cases}
& 0\;\;\text{when $r=1$}\\
& 1\;\;\text{when $r\geq 2$}\\
\end{cases} ,\;\;\;\;
\lambda_r(0,1)=\begin{cases}
& 0\;\; \text{when $r$ is even}\\
& 1\;\; \text{when $r$ is odd}
\end{cases}
.\end{equation}

\[\mu_{r,5}(6,-5)
=\begin{cases}
& 2\;\;\text{when $r$ is odd}\\
& 0\;\;\text{when $r$ is even}\\
\end{cases} ,\;
\lambda_r(6,-5)=\begin{cases}
&6\;\; \text{when $r$ is even}\\
&-5\;\; \text{when $r$ is odd}
\end{cases}.
\]
Now using $\theta_5(\lambda,\mu)$ values from Table \ref{t:3.5} we have
\[A_{2r}(0,1)=2r-1\;\;\;,A_{2r}(6,-5)=2r.\] Therefore using theorem \ref{T:1.1} we have that for all $m\geq 1$ and $r\geq 1$,
\begin{equation}\label{E:5.7}
\begin{aligned}
p_{[1^05^1]}\bigg(5^{2r}m+\frac{5-5^{2r+1}}{24}\bigg)\equiv 0\pmod{5^{2r-1}},\\
p_{[1^65^{-5}]}\bigg(5^{2r}m+\frac{19\cdot5^{2r}-19}{24}\bigg)\equiv 0\pmod{5^{2r}}.
\end{aligned}
\end{equation}
Now using \eqref{E:5.4} and the fact that \eqref{E:5.7} is true when $m$ replaced by $m+1$, we can see that \eqref{E:1.6} is true.\\

Now we prove \eqref{E:1.7}. From \eqref{E:1.3} we have
\begin{equation}\label{E:5.8}
C\Phi_7(q)=\sum_{n=0}^{\infty}\bigg(p_{[1^07^1]}(n)+7^2p_{[1^47^{-3}]}(n-1)+7^3p_{[1^87^{-7}]}(n-2)\bigg)q^n.
\end{equation}
Now using \eqref{E:n_r} we have that
\begin{equation}\label{E:5.9}
\begin{aligned}
    n_{2r,7}(0,1)=&\frac{7-7^{2r+1}}{24},\;\;\;n_{2r,5}(4,-3)=\frac{17\cdot 7^{2r}-17}{24},\\
    n_{2r,7}(8,-7)=&\frac{41\cdot 7^{2r}-41}{24}.
    \end{aligned}
\end{equation}
To find $\mu_r$ we use \eqref{E:lr} and \eqref{E:mur}. Thus we have that, 
\begin{equation}\label{E:5.10}
\mu_{r,7}(0,1)
=\begin{cases}
& 0\;\;\text{when $r=1$}\\
& 1\;\;\text{when $r\geq 2$}\\
\end{cases} ,\;\;\;\;
\lambda_r(0,1)=\begin{cases}
& 0\;\; \text{when $r$ is even}\\
& 1\;\; \text{when $r$ is odd}
\end{cases}
.\end{equation}

\[\mu_{r,7}(4,-3)
=\begin{cases}
& 2\;\;\text{when $r$ is odd}\\
& 0\;\;\text{when $r$ is even}\\
\end{cases} ,\;
\lambda_r(4,-3)=\begin{cases}
&4\;\; \text{when $r$ is even}\\
&-3\;\; \text{when $r$ is odd}
\end{cases}
.\]
\[\mu_{r,7}(8,-7)
=\begin{cases}
& 3\;\;\text{when $r$ is odd}\\
& 0\;\;\text{when $r$ is even}\\
\end{cases} ,\;
\lambda_r(8,-7)=\begin{cases}
&8\;\; \text{when $r$ is even}\\
&-7\;\; \text{when $r$ is odd}
\end{cases}
.\]\\

Now using \eqref{E:5.9} and the values of $\theta_7(\lambda,\mu)$  from the table \ref{t:3.6} we have
\begin{equation}\label{E:5.11}
A_{2r}(0,1)=r,\;\;A_{2r}(4,-3)=r,\;\;A_{2r}(8,-7)=r.
\end{equation}

Therefore using \eqref{E:5.11} and the theorem \ref{T:1.1} we have that for all $m\geq 1$ and $r\geq 1$,
\begin{equation}\label{E:5.12}
\begin{aligned}
p_{[1^07^1]}\bigg(7^{2r}m+\frac{7-7^{2r+1}}{24}\bigg)&\equiv 0\pmod{7^{r}},\\
p_{[1^47^{-3}]}\bigg(7^{2r}m+\frac{17\cdot7^{2r}-17}{24}\bigg)&\equiv 0\pmod{7^{r}},\\
p_{[1^87^{-7}]}\bigg(7^{2r}m+\frac{41\cdot7^{2r}-41}{24}\bigg)&\equiv 0\pmod{7^{r}}.
\end{aligned}
\end{equation}

Now using \eqref{E:5.8} and the fact that \eqref{E:5.12} is true when $m$ replaced by $m+1$ and $m+2$, we can see that \eqref{E:1.7} is true.\\

Now we prove \eqref{E:1.8}. From \eqref{E:1.4} we have
\begin{equation}\label{E:5.13}
C\Phi_{11}(q)=\sum_{n=0}^{\infty}\bigg(p_{[1^011^1]}(n)+11\cdot p(11n-5)\bigg)q^n.
\end{equation}
Now using \eqref{E:n_r} we have that
\begin{equation}\label{E:5.14}
    n_{[2r,11]}(0,1)=\frac{11-11^{2r+1}}{24},\;\;\;n_{[2r+1,11]}(1,0)=\frac{-11^{2r+1}+1}{24}
\end{equation}
To find $\mu_r$ we use \eqref{E:lr} and \eqref{E:mur}. Thus we have that
\begin{equation}\label{E:5.15}
\mu_{[r,11]}(0,1)
=\begin{cases}
& 0\;\;\text{when $r=1$}\\
& 1\;\;\text{when $r\geq 2$}\\
\end{cases} ,\;\;\;\;
\lambda_r(0,1)=\begin{cases}
& 0\;\; \text{when $r$ is even}\\
& 1\;\; \text{when $r$ is odd}
\end{cases}
.\end{equation}

\[\mu_{[r,11]}(1,0)
=1 \text{for all $r$},\;\;\;
\lambda_r(4,-3)=\begin{cases}
&1\;\; \text{when $r$ is even}\\
&0\;\; \text{when $r$ is odd}
\end{cases}
.\]

Now using \eqref{E:5.9} and the values of $\theta_{11}(\lambda,\mu)$  from the table $3$ in \cite{P} we have
\begin{equation}\label{E:5.16}
A_{2r}(0,1)=2r-1,\;\;A_{2r}(1,0)=2r.
\end{equation}\\

Therefore using \eqref{E:5.16} and the theorem \ref{T:1.1} we have that for all $m\geq 1$ and $r\geq 1$,
\begin{equation}\label{E:5.17}
\begin{aligned}
p_{[1^011^1]}\bigg(11^{2r}m+\frac{11-11^{2r+1}}{24}\bigg)&\equiv 0\pmod{11^{2r}},\\
p\bigg(11^{2r+1}m+\frac{1-11^{2r+2}}{24}\bigg)&\equiv 0\pmod{11^{2r+1}}.
\end{aligned}
\end{equation}

Now using \eqref{E:5.13} and \eqref{E:5.17}, we can see that \eqref{E:1.8} is true.\\

Now we prove \eqref{E:1.9}. From \eqref{E:1.5} we have
\begin{equation}\label{E:5.18}
C\Phi_{13}(q)\equiv \sum_{n=0}^{\infty}p_{[1^013^1]}(n)q^n\pmod{13}.\end{equation}

For $c=0,d=1$, we see that using Table \ref{t:3.7}, $\theta(\lambda_1,\mu_1)=0$  by equation \eqref{E:n_r} we have that
\[n_{1}(0,1)=0.\]
Hence by corollary \ref{C1.5} we have for some $m$,
\[p_{[1^013^1]}\bigg(13^{2r}m-\frac{13-13^{2r+1}}{24}\bigg)\not\equiv 0\pmod{13}.\]
\end{proof}

\section{Other examples}
\subsection{Congruences for $\ell-$ regular partitions}
Following \cite{WL2}, the generating function for this partition function is,
\begin{equation}
    \sum_{n=0}^{\infty}b_{\ell}(n)q^n=\prod_{n=1}^{\infty}\frac{(1-q^{\ell n})}{(1-q^n)}.
\end{equation}
In \cite{WL},\cite{WL2} and \cite{WL3} Wang proved  congruences for the $\ell-$ regular partitions for $\ell=5,7$ and $11$. Here we used our method to quickly conclude these congruences for primes $5,7,11$ and incongruences for primes $13$ and $17$.

\begin{corollary}\label{6.1} For any integer $m\geq0$ and for each positive integer $r$,
\begin{equation*}
    \begin{aligned}
        b_{5}\left(5^{2r}m+\frac{5^{2r}-1}{6}\right)&\equiv0\pmod{5^r},\\
        b_{7}\left(7^{2r}m+\frac{\cdot7^{2r}-1}{4}\right)&\equiv0\pmod{7^r},\\
        b_{11}\left(11^{2r}m+5\cdot\frac{11^{2r}-1}{12}\right)&\equiv0\pmod{11^r}.\\
    \end{aligned}
\end{equation*}
\end{corollary}

\begin{corollary}\label{6.2} For each positive integer $r$ and for some integer $m$,
\begin{equation*}
\begin{aligned}
  b_{13}\left(13^{2r}m+\frac{13^{2r}-1}{2}\right)&\not\equiv0\pmod{13},\\
  b_{17}\left(17m-12\right)&\not\equiv0\pmod{17}.
        \end{aligned}
        \end{equation*}
\end{corollary}

\begin{proof}
The proof of the corollaries \ref{6.1} and \ref{6.2} follows from the entries of the table \ref{t:6.1}. Calculation of $A_r$ follows from table $3$ in \cite{P} when $\ell=11$.
\begin{table}[htp]
\begin{center}
  \begin{tabular}{ |p{1cm}||p{2cm}|p{4cm}|p{2cm}|  }
 \hline
 $\ell$ &$n_{2r}$& $\mu_{r}$  & $A_r$ \\
 \hline\hline
 $5$  &$\frac{5^{2r}-1}{6}$ & $\mu_{2r-2}=0$, $\mu_{2r-1}=1$ &  $A_{2r}=r$  \\
\hline
$7$ &$\frac{3\cdot7^{2r}-1}{4}$&  $\mu_{2r-2}=0$, $\mu_{2r-1}=1$  & $A_{2r}=r$ \\
\hline
$11$  &$\frac{5\cdot11^{2r}-5}{12}$&  $\mu_{2r-2}=0$, $\mu_{2r-1}=1$  & $A_{2r}=r$  \\
\hline
$13$  &$n_1=-7$ & $\mu_{2r-2}=0$, $\mu_{2r-1}=1$  & $A_r=0$  \\
\hline
$17$  & $n_1=-12$&$\mu_{2r-2}=0$, $\mu_{2r-1}=1$    & $A_r=0$  \\
\hline
\end{tabular}
\vspace{2mm}

\caption{Calculations for $\ell-$regular partitions}\label{t:6.1}
\end{center}
\end{table}
\end{proof}

\subsection{Congruences for $\ell-$core partitions}

Following \cite{WL}, the generating function for this partition function is,
\begin{equation}
    \sum_{n=0}^{\infty}a_{\ell}(n)q^n=\prod_{n=1}^{\infty}\frac{(1-q^{\ell n})^{\ell}}{(1-q^n)}.
\end{equation}
 Combintorially, $\ell-$core partitions are defined as the partitions of $n$ where hook lengths of each entry in the Ferres diagram are not divisible by $\ell$. \\ 

\begin{corollary}\label{6.3} For integers $m\geq 0$ and for each positive integer $r$, 
\begin{equation*}
    \begin{aligned}
        a_{5}\left(5^{r}m-1\right)&\equiv0\pmod{5^r},\\
        a_{7}\left(7^{r}m-2\right)&\equiv0\pmod{7^{r}},\\
        a_{11}\left(11^rm-5\right)&\equiv0\pmod{11^r}
    \end{aligned}
\end{equation*}
\end{corollary}

\begin{corollary}\label{6.4} For each positive integer $r$ and for some integer $m$,
\begin{equation*}
\begin{aligned}
    a_{13}\left(13^{2r}m-7\right)&\not\equiv0\pmod{13},\\
 a_{17}\left(17m-12\right)&\not\equiv0\pmod{17}.  
\end{aligned}
\end{equation*}
\end{corollary}

\begin{proof}
Similarly, the proof of the corollaries \ref{6.3} and \ref{6.4} follows from the entries of table \ref{t:6.2} , calculation of $A_r$ follows from table $3$ in \cite{P} when $\ell=11$.

\begin{table}

\centering
 \begin{tabular}{ |p{1cm}||p{2cm}|p{5.5cm}|p{2cm}|  }
 \hline
 $\ell$ &$n_{r}$& $\mu_{r}$  & $A_r$ \\
 \hline\hline
 $5$  &$n_{r}=-1$ & $\mu_{2r}=0, \mu_{2r-1}=1$ &  $A_{2r}=r$  \\
\hline
$7$ &$n_{r}=-2$&  $\mu_{2r}=0, \mu_{2r-1}=1$  & $A_{r}=r$ \\
\hline
$11$  &$n_{r}=-5$&  $\mu_0=0, \mu_{2r}=-4, \mu_{2r+1}=1$  & $A_{r}=r$  \\
\hline
$13$  &$n_{2r}=-7$ & $\mu_0=0, \mu_{2r}=-6, \mu_{2r-1}=1$  & $A_r=0$  \\
\hline
$17$  & $n_{1}=-12$&$ \mu_{0}=0,\mu_{2r}=-11, \mu_{2r-1}=1$  & $A_r=0$  \\
\hline
\end{tabular}
\vspace{2mm}
\caption{Calculations for $\ell-$core partitions.}\label{t:6.2}
\end{table}
\end{proof}

\section*{Acknowledgements}
The author thanks thesis advisor Karl Mahlburg for his guidance during this project. The author also thanks Fang-Ting Tu for helpful discussions and Jeremy Lovejoy for helpful comments on the first version of the paper.

\end{document}